\documentclass[11pt]{article}
\usepackage{geometry}
\usepackage{amssymb, amsmath, amsthm, color, tikz,enumerate,tikz-3dplot}

\usepackage{fullpage}

\newtheorem{theorem}{Theorem}[section]
\newtheorem{proposition}[theorem]{Proposition}

\newtheorem{lemma}[theorem]{Lemma}

\newtheorem{corollary}[theorem]{Corollary}
\newtheorem{remark}[theorem]{Remark}

\numberwithin{equation}{section}

\newcommand{\Bbbb}{\mathbb{B}}

\newcommand{\Rrr}{\mathbb{R}}
\newcommand{\Sss}{\mathbb{S}}
\newcommand{\Qqq}{\mathbb{Q}}

\newcommand{\Sym}{\mathfrak{S}}
\newcommand{\R}{\Rrr}

\DeclareMathOperator{\cross}{cross}

\newcommand{\hz}{\hat{0}}
\newcommand{\ho}{\hat{1}}
\newcommand{\coveredby}{\prec}

\newcommand{\Comp}{\operatorname{Comp}}
\newcommand{\Des}{\operatorname{Des}}
\newcommand{\id}{\operatorname{id}}
\newcommand{\inv}{\operatorname{inv}}
\newcommand{\type}{\operatorname{type}}

\newcommand{\cupdots}{\cup \cdots \cup}

\newcommand{\Qn}{\Pi^{\text{ord}}_{n}}

\newcommand{\vanish}[1]{}

\newcommand{\fl}{\rightarrow}
\newcommand{\ungras}{1\!\!\mkern -1mu1}
\newcommand{\iso}{\stackrel{\sim}{\fl}}

\newcommand{\lambdaweaklydecreasing}
{\lambda_{1} \geq \lambda_{2} \geq \cdots \geq \lambda_{n}}
\newcommand{\lambdaweaklyincreasing}
{\lambda_{1} \leq \lambda_{2} \leq \cdots \leq \lambda_{n}}

\newcommand{\lambdasumpositive}
{\lambda_{1} + \lambda_{2} + \cdots + \lambda_{n} > 0}
\newcommand{\lambdasumnonpositive}
{\lambda_{1} + \lambda_{2} + \cdots + \lambda_{n} \leq 0}

\begin{document}

\title{Some combinatorial identities appearing in
the calculation of the cohomology of Siegel modular varieties}

\author{{\sc Richard EHRENBORG},
        {\sc Sophie MOREL}
         and
        {\sc Margaret READDY}}

\date{}

\maketitle

\begin{abstract}
In the computation of the intersection cohomology of Shimura varieties,
or of the $L^2$ cohomology of equal rank locally symmetric spaces, 
combinatorial identities involving averaged discrete series characters of
real reductive groups play a large
technical role.
These identities can become very complicated and are
not always well-understood (see for example the appendix of~\cite{Morel}).
We propose a geometric approach to these
identities in the case of Siegel modular varieties
using the combinatorial
properties of the Coxeter complex of the symmetric group.
Apart from some introductory remarks
about the origin of the identities,
our paper is entirely combinatorial and does not require any knowledge of
Shimura varieties or of representation theory.

\vspace*{2 mm}

\noindent
{\em 2010 Mathematics Subject Classification.}
Primary 14G35, 05E45;
Secondary 14F43, 05A18, 06A07, 06A11, 52B22.


\vspace*{2 mm}

\noindent
{\em Keywords and phrases.}
Averaged discrete series characters,
permutahedron,
intersection cohomology,
ordered set partitions,
shellability.
\end{abstract}

\section{Introduction}

The goal of this paper is to give more natural and geometric proofs
of some combinatorial identities that appear when one calculates
the commuting actions of the Hecke algebra and the absolute Galois group
of~$\Qqq$ on
the cohomology of a Siegel modular variety; actually, the cohomology that is
used is
the intersection
cohomology of the minimal compactification. These identities, which
appear in the calculation of weighted orbital integrals at the real place,
are the technical heart of the paper~\cite{Morel}, but 
were relegated to an appendix and proved by brute force.

For the expert,
we next provide some details about the manner our combinatorial identities
appear in~\cite{Morel}.
The calculation in Proposition~3.3.1 of
\cite{Morel}, which is the central technical result of that
paper, requires us to identify
two virtual representations of the real points of a maximal torus $T$ of
the general symplectic group $G$. The first expression, called
$L_M(\gamma_M)$ in~\cite[Section 1.2]{Morel},
comes from a
geometric calculation using the Lefschetz fixed point formula, and involves
the action of $T$ on truncated cohomology groups of the Lie algebra of the
unipotent radical of parabolic subgroups containing $T$. 
The second
expression, which comes from Arthur and Kottwitz's expression for the
spectral side of the stable trace formula, involves averaged characters of
discrete series representations of $G(\R)$; 
see~\cite[Section 3.1]{Morel}. Both expressions can be reformulated
as linear combinations of quasi-characters of $T(\R)$, the first  via
formulas of Kostant for the Lie algebra cohomology and Weyl for the character
of an algebraic representation, and the second via Harish-Chandra's
formula for discrete series characters, made more explicit by Herb
in~\cite{Herb}; this is explained in the proof of Proposition~3.3.1
of~\cite{Morel}.
When we equate the coefficients of a fixed quasi-character
on both sides, we arrive at Corollary~A.5 of~\cite{Morel}, which follows
easily from Proposition~A.4 of that article.  The proof of that
proposition in~\cite{Morel} is long, technical and not exactly enlightening.

In this paper, we present a
geometric proof of this result (our Theorem~\ref{theorem_S_equal_T})
which involves the geometry of the
Coxeter complex of the symmetric group~$\Sym_{n}$.
This complex can also be described as the boundary
of the dual polytope of the permutahedron.
This geometric viewpoint also yields
a strengthening of Corollary~A.3 of~\cite{Morel}.
We emphasize that our approach is completely combinatorial.
Although our identities 
originally stem from the representation theory and 
arithmetic geometry used in the calculation 
of the cohomology of Shimura varieties,
the current paper is directed to a combinatorial audience.

An overview of the paper is as follows.
Section~\ref{section_preliminaries}
contains preliminaries
about the permutahedron
and shellings.
Section~\ref{section_the_weighted_complex}
gives the definition of the weighted subcomplex $\Sigma(\lambda)$ of
the Coxeter complex~$\Sigma_{n}$ of the symmetric group~$\Sym_n$
that we
wish to study.  We prove $\Sigma(\lambda)$ is a pure 
subcomplex of the same
dimension as $\Sigma_{n}$;
see Lemma~\ref{lemma_pure}.
In
Section~\ref{section_shellings_of_Coxeter_complex}
we give a brief proof of the theorem (originally
due to Bj\"orner; see~\cite{Bjorner_II}) 
that any linear extension of the weak Bruhat order is
a shelling order on the facets of $\Sigma_{n}$, 
and deduce a similar result
for~$\Sigma(\lambda)$.
Section~\ref{section_lexicographic_shelling}
yields another proof that
the weighted complex $\Sigma(\lambda)$ is shellable
by viewing it as the order complex of an
$EL$-shellable poset.
The shelling results of
Sections~\ref{section_shellings_of_Coxeter_complex}
and~\ref{section_lexicographic_shelling}
imply that
$\Sigma(\lambda)$ is always homeomorphic to a ball or a sphere, and imply
Corollary~A.3 of~\cite{Morel}, but they are much stronger than this
corollary.
In Section~\ref{section_main_result}
we state our main result, Theorem~\ref{theorem_S_equal_T},
which corresponds to Proposition~A.4 of~\cite{Morel}.
Sections~\ref{section_base_case}
and~\ref{section_permute}
contain the proof of this theorem:
in Section~\ref{section_base_case}
we give the proof in the case
$\lambda$ is weakly increasing,
and in
Section~\ref{section_permute}
we show how to reduce the general case
to this base case. Finally,
in Section~\ref{section_expression}
we derive another expression for the left-hand side of the identity of
Theorem~\ref{theorem_S_equal_T}
when $\lambda$ is a weakly decreasing sequence.

\section{Preliminaries}
\label{section_preliminaries}

\subsection{Permutations and ordered partitions}

Let $\Sym_{n}$ denote the symmetric group on $n$ elements.
We write permutations $\tau \in \Sym_{n}$ in one-line notation, that is,
$\tau = \tau_{1} \tau_{2} \cdots \tau_{n}$.
Let the symmetric group $\Sym_{n}$ act upon the
vector space~$\Rrr^{n}$ by permuting the coordinates,
that is,
given a permutation $\tau$
and the vector $x = (x_{1}, x_{2}, \ldots, x_{n})$,
we define
$\tau(x) = (x_{\tau^{-1}(1)}, x_{\tau^{-1}(2)}, \ldots, x_{\tau^{-1}(n)})$.
This is a left action, since
for two permutations $\tau$ and $\pi$ we have that
$\tau(\pi(x)) = (\tau \circ \pi)(x)$.

Define $[n]$ to be the set $\{1,2, \ldots, n\}$.
We will need the following permutation statistics.
The {\em descent set} of a permutation $\tau \in \Sym_{n}$
is the set
$\Des(\tau) = \{i \in [n-1] \: : \: \tau_{i} > \tau_{i+1}\}$.
The {\em descent composition} is the list
$(d_{1}-d_{0}, d_{2}-d_{1}, \ldots, d_{k+1}-d_{k})$
where $\Des(\tau) = \{d_{1} < d_{2} < \cdots < d_{k}\}$
and we tacitly assume $d_{0} = 0$ and $d_{k+1} = n$.
The {\em number of inversions} of a permutation, also called the
\emph{length},
is given by
$\inv(\tau) = |\{(i,j) \: : \: 1 \leq i < j \leq n, \tau_{i} > \tau_{j}\}|$.
Finally, denote the {\em sign} of a permutation $\tau$ by
$(-1)^{\tau}$, that is,
$(-1)^{\tau} = (-1)^{\inv(\tau)}$.

In the symmetric group $\Sym_{n}$
let $s_{i}$ denote the simple transposition $(i,i+1)$
where $1 \leq i \leq n-1$.
The {\em weak Bruhat order} on the symmetric group $\Sym_{n}$
is defined by the cover relation
$\tau \coveredby \tau s_{i}$ where
$\inv(\tau) < \inv(\tau s_{i})$.
More explicitly, the cover relation is
$$
\tau_{1} \cdots \tau_{i} \tau_{i+1} \cdots \tau_{n}
\coveredby
\tau_{1} \cdots \tau_{i+1} \tau_{i} \cdots \tau_{n}
\:\:\:\:
\text{ if }
\:\:\:\:
\tau_{i} < \tau_{i+1} .
$$
With respect to this partial order,
the identity element $12 \cdots n$ is the minimal element,
the permutation $n \cdots 21$ is the maximal element,
and the rank function is given by the number of inversions
of the permutation.

The $(n-1)$-dimensional
{\em permutahedron} 
is the simple polytope defined by
taking the convex hull of the $n!$ points 
$(\tau_{1}, \tau_{2}, \ldots, \tau_{n}) \in \Rrr^{n}$
where $\tau \in \Sym_{n}$.
These $n!$ points lie in the hyperplane
$\sum_{i=1}^{n} x_{i} = \binom{n+1}{2}$,
and one can verify
they generate a
polytope of dimension $n-1$.
Label the vertex
$(\tau_{1}, \tau_{2}, \ldots, \tau_{n})$ by the inverse permutation $\tau^{-1}$.
The $1$-skeleton of the $(n-1)$-dimensional permutahedron
yields the Hasse diagram of the weak Bruhat order of $\Sym_{n}$,
where the cover relations are oriented away from
the vertex labeled $12 \cdots n$ and toward the vertex labeled $n \cdots 21$.

A {\em partition} $\pi = \{B_{1}, B_{2}, \ldots, B_{k}\}$
of the set $[n]$ is a collection of subsets of~$[n]$, called {\em blocks},
such that $B_{i} \neq \emptyset$ for $1 \leq i \leq k$,
$B_{i} \cap B_{j} = \emptyset$ for $1 \leq i < j \leq k$
and
$\bigcup_{i=1}^{k} B_{i} = [n]$.
Let $\Pi_{n}$ denote
the collection of all the partitions of the set $[n]$.
We make $\Pi_{n}$ into a partially ordered set (poset)
by the cover relation
$$
  \{B_{1}, B_{2}, B_{3}, \ldots, B_{k}\}
  \coveredby
  \{B_{1} \cup B_{2}, B_{3}, \ldots, B_{k}\} .
$$
In other words, $\Pi_{n}$ is ordered by reverse refinement.
The poset $\Pi_{n}$ is in fact a lattice,
known as the partition lattice.
Furthermore, the partition lattice $\Pi_{n}$ is a graded poset,
with minimal element
$\hz = \{\{1\}, \{2\}, \ldots, \{n\}\}$,
maximal element
$\ho = \{[n]\}$
and rank function
$\rho(\pi) = n-|\pi|$,
where
$|\pi|$ denotes the number of blocks of the partition $\pi$.

An {\em ordered partition} 
$\sigma = (C_{1}, C_{2}, \ldots, C_{k})$
of the set $[n]$
is a list of subsets of $[n]$ such that
$\{C_{1}, C_{2}, \ldots, C_{k}\}$
is a partition of $[n]$.
Let $\Qn$ denote the set of all
ordered partitions on the set~$[n]$.
The set $\Qn$ forms a poset
by letting the cover relation
be the merging of two adjacent blocks, that is,
$$
(C_{1}, \ldots, C_{i}, C_{i+1}, \ldots, C_{k})
\coveredby
(C_{1}, \ldots, C_{i} \cup C_{i+1}, \ldots, C_{k}) .
$$
Observe that the maximal element of
$\Qn$ is
the ordered partition consisting of one block~$([n])$.
However there are $n!$ minimal elements,
one for each permutation $\tau = \tau_{1} \tau_{2} \cdots \tau_{n}$
in the symmetric group $\Sym_{n}$, namely
the ordered partitions of the form
$(\{\tau_{1}\}, \{\tau_{2}\}, \ldots, \{\tau_{n}\})$.
We identify these minimal elements with
permutations in~$\Sym_{n}$ written in one-line notation.
Let $|\sigma|$ denote the number of blocks of the ordered partition $\sigma$.
Also observe that every interval in $\Qn$ is isomorphic
to a Boolean algebra,
that is,
the interval $[\sigma_{1},\sigma_{2}]$
is isomorphic to the Boolean algebra $B_{|\sigma_{1}| - |\sigma_{2}|}$.
For recent work regarding ordered partitions, 
see~\cite{Ehrenborg_Hedmark,Ehrenborg_Jung}.

When we adjoin a minimal element $\hz$ to $\Qn$
the resulting poset is a lattice,
called the {\em ordered partition lattice}. In fact, it is
the face lattice of the $(n-1)$-dimensional permutahedron.

A {\em composition} of $n$ is a list 
$(c_{1}, c_{2}, \ldots, c_{k})$ of positive integers
such that $\sum_{i=1}^{k} c_{i} = n$.
Let $\Comp(n)$ denote the set of all compositions of $n$.
The set $\Comp(n)$ forms a poset by letting
the cover relation be the adding of two adjacent entries,
that is,
$$
(c_{1}, \ldots, c_{i}, c_{i+1}, \ldots, c_{k})
\coveredby
(c_{1}, \ldots, c_{i} + c_{i+1}, \ldots, c_{k}) .
$$
The poset $\Comp(n)$ is isomorphic to the Boolean algebra $B_{n-1}$.
Define the map $\type : \Qn \longrightarrow \Comp(n)$
by reading off the cardinalities of the blocks, that is,
$$ \type((C_{1}, C_{2}, \ldots, C_{k})) = (|C_{1}|, |C_{2}|, \ldots, |C_{k}|) . $$
This is an order-preserving map. 

\subsection{Shellable and decomposable simplicial complexes}

For a face $F$ of a simplicial complex $\Delta$
let $\overline{F}$ denote the subcomplex
$\{G \: : \: G \subseteq F\}$.
Recall that a pure simplicial complex $\Delta$ 
of dimension $d$ is said to be
{\em shellable} if 
it is either $0$-dimensional, or
there is an ordering of the facets $F_{1}, F_{2}, \ldots, F_{s}$
such that  the complex
$\overline{F_{j}} \cap \left(\bigcup_{i=1}^{j-1} \overline{F_{i}}\right)$
is a pure simplicial complex of dimension $d-1$ for $2 \leq j \leq s$.
As a remark, there are more general formulations of 
shellability for non-pure
complexes and polytopal complexes; see~\cite{Bjorner_Wachs}.
For a facet $F_{j}$ let $R(F_{j})$ be the 
facet restriction, that is, the smallest dimensional face
of $F_{j}$ that does not appear in the complex
$\bigcup_{i=1}^{j-1} \overline{F_{i}}$.
The shelling condition implies
that the face poset of $\Delta$ can be written as the disjoint union
$\bigcup_{j=1}^{s} [R(F_{j}), F_{j}]$.
We call the facet $F_{j}$ a {\em homology facet}
if $R(F_{j}) = F_{j}$.
It is only for homology facets that the topology
of the complex 
changes in the $j$th shelling step from
$\overline{F_{1}} \cupdots \overline{F_{j-1}}$
to
$\overline{F_{1}} \cupdots \overline{F_{j}}$.
The subcomplex
$\overline{F_{1}} \cupdots \overline{F_{j}}$
is said to be a
{\em partial shelling} 
of the complex~$\Delta$.
See~\cite{Bjorner,Stanley_green} and its references for background on shellability.

Recall that a pure simplicial complex $\Delta$ 
is {\em decomposable} if 
every facet $F_{j}$ has a face $R(F_{j})$ such that
we can write $\Delta$ as
a disjoint union
\begin{align}
\Delta & = \bigcup_{j=1}^{s} [R(F_{j}), F_{j}] .
\label{equation_decomposition}
\end{align}
Complexes satisfying this property are usually called
\emph{partitionable} (see~\cite{Stanley_green}), but we use the term ``decomposable''
here to avoid any confusion with ordered partitions.

A shellable complex is decomposable,
but the converse is not true in general.
However, the following lemma, which follows
easily from Proposition~2.5 of
\cite{Bjorner_Wachs},
yields a condition
which implies shellability.

\begin{lemma}
Let $\Delta$ be a decomposable simplicial complex
with the decomposition
given by~\eqref{equation_decomposition}.
The ordering of the facets
$F_{1}, F_{2}, \ldots, F_{s}$ is a shelling order
if and only if for every face $G$ of the facet $F_{j}$
there exist an index $i \leq j$
such that the face $G$ belongs to
the interval $[R(F_{i}), F_{i}]$.
\label{lemma_shelling_condition}
\end{lemma}

\vanish{
\begin{proof}
We proceed by induction on $j$.
When $j=1$ the condition implies that
$R(F_{1})$ is the empty face, and hence
$\overline{F_{1}}$ is a shellable complex.
Assume now that
$\overline{F_{1}} \cupdots \overline{F_{j-1}}$
is a shellable complex,
that is,
$\overline{F_{1}} \cupdots \overline{F_{j-1}}
=
\bigcup_{i=1}^{j-1} [R(F_{i}), F_{i}]$.
Then any face of $F_{j}$ is either already a face
of $\overline{F_{1}} \cupdots \overline{F_{j-1}}$ 
or a face of $\overline{F_{j}} - \overline{F_{1}} \cupdots \overline{F_{j-1}}$ 
Hence the intersection
$(\overline{F_{1}} \cupdots \overline{F_{j-1}}) \cap \overline{F_{j}}$
is the complement of the interval $[R(F_{j}), F_{j}]$ in~$\overline{F_{j}}$,
proving that
$F_{1}, F_{2}, \ldots, F_{j}$ is a shelling order
and completing the induction.
\end{proof}
}

For further background on combinatorial structures and posets,
see~\cite{Stanley_EC_1}.
For more on the combinatorics of simplicial complexes,
see~\cite{Stanley_green}.

\section{The weighted complex $\Sigma(\lambda)$}
\label{section_the_weighted_complex}

Let $\lambda$ be a sequence of $n$ real numbers,
that is, $\lambda = (\lambda_{1}, \lambda_{2}, \ldots, \lambda_{n}) \in \Rrr^{n}$.
For a subset $S \subseteq [n]$ we introduce the shorthand notation
$\lambda_{S} = \sum_{i \in S} \lambda_{i}$.
Define the subset $\mathcal{P}(\lambda)$ of the set of ordered 
partitions~$\Qn$ by
$$
\mathcal{P}(\lambda)
=
\left\{ \sigma = (C_{1}, C_{2}, \ldots, C_{k}) \in \Qn \:\: : \:\:
\sum_{i=1}^{j} \lambda_{C_{i}} > 0 \text{ for } 1 \leq j \leq k \right\} .
$$
Note that if $\lambdasumnonpositive$
then the set $\mathcal{P}(\lambda)$ is empty.

\begin{lemma}
The set $\mathcal{P}(\lambda)$ is an upper order ideal (also
know as a filter)
in the poset $\Qn$.
\end{lemma}
This follows directly from the definitions
since by merging two adjacent blocks
there is one less inequality to verify.

\begin{lemma}
Given an ordered partition
$\sigma = (C_{1}, C_{2}, \ldots, C_{k}) \in \mathcal{P}(\lambda)$,
assume that $C_{j}$ is a non-singleton block.
Let $a$ be an element of the block $C_{j}$
with maximal $\lambda$-value,
that is,
$\lambda_{a} = \max_{b \in C_{j}}(\lambda_{b})$.
Let $\sigma^{\prime}$ be the ordered partition
$$
\sigma^{\prime}
=
(C_{1}, C_{2}, \ldots, C_{j-1}, \{a\}, C_{j} - \{a\}, C_{j+1}, \ldots, C_{k}) .
$$
Then the cover relation $\sigma^{\prime} \coveredby \sigma$ holds.
Furthermore, $\sigma^{\prime}$ belongs to the set $\mathcal{P}(\lambda)$.
\label{lemma_split}
\end{lemma}
\begin{proof}
The cover relation $\sigma^{\prime} \coveredby \sigma$ is immediate.
To verify that
$\sigma^{\prime} \in \mathcal{P}(\lambda)$
it is enough to verify that
$\sum_{i=1}^{j-1} \lambda_{C_{i}} + \lambda_{a} > 0$.
If $\lambda_{a} \geq 0$ this is immediately true.
If $\lambda_{a} < 0$ then
the inequality follows by noticing that all the 
$\lambda$-values associated to
the elements in the block $C_{j}$
are negative and we have that
$\sum_{i=1}^{j-1} \lambda_{C_{i}} + \lambda_{a} \geq
\sum_{i=1}^{j-1} \lambda_{C_{i}} + \lambda_{C_{j}} > 0$.
\end{proof}

Similar to the previous lemma we have the next result.
This lemma will be used in the proof of the main result
in Section~\ref{section_permute}.

\begin{lemma}
Given an ordered partition
$\sigma = (C_{1}, C_{2}, \ldots, C_{k}) \in \mathcal{P}(\lambda)$,
assume that $B$ is a non-empty proper subset of $C_{j}$,
that is,
$\emptyset \subsetneq B \subsetneq C_{j}$.
Let $\sigma^{\prime}$ be the ordered partition
$$
\sigma^{\prime}
=
\begin{cases}
(C_{1}, C_{2}, \ldots, C_{j-1}, B, C_{j} - B, C_{j+1}, \ldots, C_{k}) 
& \text{ if } \lambda_{B} > 0 , \\
(C_{1}, C_{2}, \ldots, C_{j-1}, C_{j} - B, B, C_{j+1}, \ldots, C_{k}) 
& \text{ if } \lambda_{B} \leq 0 . 
\end{cases}
$$
Then the cover relation $\sigma^{\prime} \coveredby \sigma$ holds.
Furthermore, $\sigma^{\prime}$ belongs to the set $\mathcal{P}(\lambda)$.
\label{lemma_split_block}
\end{lemma}
\begin{proof}
Again, the cover relation is immediate.
The fact that $\sigma^{\prime} \in \mathcal{P}(\lambda)$
follows by the two cases:
In the case $\lambda_{B} > 0$
it is enough to observe that
$\sum_{i=1}^{j-1} \lambda_{C_{i}} + \lambda_{B} > 0$.
In the second case, we have
$\sum_{i=1}^{j-1} \lambda_{C_{i}} + \lambda_{C_{j} - B}
=
\sum_{i=1}^{j} \lambda_{C_{i}} - \lambda_{B}
 > 0$.
\end{proof}

Let $\mathcal{A}(\lambda)$ denote the subset
of the symmetric group $\Sym_{n}$
defined by
$$
\mathcal{A}(\lambda)
=
\left\{ \tau \in \Sym_{n} \:\: : \:\:
\sum_{i=1}^{j} \lambda_{\tau(i)} > 0 \text{ for } 1 \leq j\leq n \right\} .
$$
Since permutations correspond to ordered partitions
having all singleton blocks,
we observe that the set
$\mathcal{A}(\lambda)$ is a subset of
$\mathcal{P}(\lambda)$. In fact, it is the subset of minimal
elements of $\mathcal{P}(\lambda)$.

\begin{lemma}
Let $\lambda \in \Rrr^{n}$ be a sequence
such that $\lambdasumpositive$.
Then the set $\mathcal{A}(\lambda)$ is nonempty.
Especially, any permutation~$\tau$
satisfying
$\lambda_{\tau_{1}} \geq \lambda_{\tau_{2}} \geq \cdots \geq \lambda_{\tau_{n}}$ 
belongs to~$\mathcal{A}(\lambda)$.
\end{lemma}
\begin{proof}
Since the sum $\sum_{i=1}^{n} \lambda_{i}$ is positive,
the ordered partition $([n])$ belongs 
to~$\mathcal{P}(\lambda)$.
Now iterating Lemma~\ref{lemma_split}
we obtain that the ordered partition corresponding to
the permutation $\tau$ 
belongs to~$\mathcal{P}(\lambda)$.
\end{proof}

\begin{lemma}
The upper order ideal $\mathcal{P}(\lambda)$
is generated by the set $\mathcal{A}(\lambda)$,
that is, for every ordered partition $\sigma \in \mathcal{P}(\lambda)$
there is a permutation $\tau \in \mathcal{A}(\lambda)$
such that $\tau \leq \sigma$ in $\Qn$.
\label{lemma_P_A}
\end{lemma}
\begin{proof}
Begin with the ordered partition $\sigma$
and iterate Lemma~\ref{lemma_split}.
This procedure yields the permutation $\tau$.
\end{proof}

Recall that $\mathcal{P}(\lambda)$
is an upper order ideal in the poset $\Qn$, and that
$\Qn$ is a join-semilattice
where each interval is isomorphic to a Boolean algebra.
Hence by reversing the order relations of~$\Qn$
we get the face poset of a simplicial
complex, and
we can view the set $\mathcal{P}(\lambda)$
as a simplicial subcomplex
$\Sigma(\lambda)$.
We call
$\Sigma(\lambda)$
the {\em weighted complex}.
See Figure~\ref{figure_lambda_II} for an example.
The maximal element $\ho$ of $\mathcal{P}(\lambda)$
is the empty face of the simplicial complex $\Sigma(\lambda)$ 
and the minimal elements $\mathcal{A}(\lambda)$ are
the facets of~$\Sigma(\lambda)$.
Note that if $\lambdasumnonpositive$ then
$\Sigma(\lambda)$ is the empty simplicial complex,
which has no faces, not even the empty face.

This idea of turning an upper order ideal upside-down
in order to view it as a simplicial complex
appears in~\cite{Ehrenborg_Hedmark,Ehrenborg_Jung}.

Lemma~\ref{lemma_P_A}
can now be reformulated as follows.
\begin{lemma}
Let $\lambda \in \Rrr^{n}$ be a sequence
such that $\lambdasumpositive$.
Then
the weighted complex~$\Sigma(\lambda)$
is a pure simplicial complex of dimension $n-2$.
\label{lemma_pure}
\end{lemma}

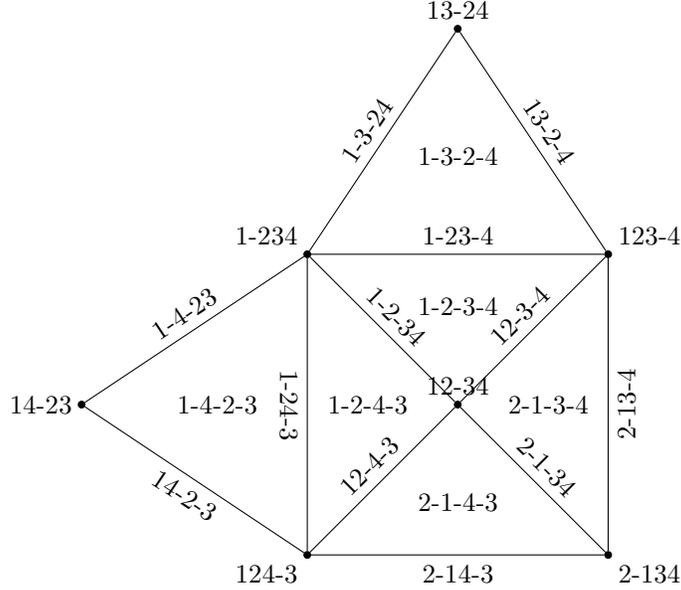
\begin{figure}
\begin{center}
\begin{tikzpicture}

\node[draw, fill=black, circle, inner sep = 0.9pt] (u) at (0, 5) {};
\node[draw, fill=black, circle, inner sep = 0.9pt] (ul) at (-2,2) {};
\node[draw, fill=black, circle, inner sep = 0.9pt] (ur) at (2,2) {};
\node[draw, fill=black, circle, inner sep = 0.9pt] (m) at (0,0) {};
\node[draw, fill=black, circle, inner sep = 0.9pt] (bl) at (-2,-2) {};
\node[draw, fill=black, circle, inner sep = 0.9pt] (br) at (2,-2) {};
\node[draw, fill=black, circle, inner sep = 0.9pt] (l) at (-5,0) {};

\draw[-]
(ul) to node [sloped,above] {\small 1-4-23}
(l) to node [sloped,below] {\small 14-2-3}
(bl) to node [sloped,below] {\small 2-14-3}
(br) to node [sloped,below] {\small 2-13-4}
(ur) to node [sloped,above] {\small 13-2-4}
(u) to node [sloped,above] {\small 1-3-24}
(ul) to node [sloped,below] {\small 1-24-3}
(bl) to node [sloped,above] {\small 12-4-3}
(m) to node [sloped,above] {\small 12-3-4}
(ur) to node [sloped,above] {\small 1-23-4}
(ul) to node [sloped,above] {\small 1-2-34}
(m) to node [sloped,above] {\small 2-1-34}
(br);

\draw (u) node [above] {\small 13-24};
\draw (ul) node [above left] {\small 1-234};
\draw (l) node [left] {\small 14-23};
\draw (bl) node [below left] {\small 124-3};
\draw (br) node [below right] {\small 2-134};
\draw (ur) node [above right] {\small 123-4};
\draw (m) node [above] {\small 12-34};

\node at (0, 1.3) {\small 1-2-3-4};
\node at (-1.2,0) {\small 1-2-4-3};
\node at (1.2,0) {\small 2-1-3-4};
\node at (0, -1.3) {\small 2-1-4-3};
\node at (-3.2,0) {\small 1-4-2-3};
\node at (0, 3.3) {\small 1-3-2-4};
\end{tikzpicture}
\end{center}
\caption{The simplicial complex $\Sigma(\lambda)$
for $\lambda = (5,1,-2,-3)$
consisting of $7$~vertices, $12$~edges and $6$~triangles.
The empty face, labeled $1234$, is not depicted.}
\label{figure_lambda_II}
\end{figure}

Reordering the entries of the sequence $\lambda$
does not change the complex as the following lemma shows.
\begin{lemma} If $\tau\in\mathfrak{S}_n$, then
the complexes $\Sigma(\lambda)$ and
$\Sigma(\tau(\lambda))$ are isomorphic
under the bijection
$\Sigma(\lambda) \longrightarrow \Sigma(\tau(\lambda))$
defined by $\sigma \longmapsto \tau(\sigma)$.
\label{lemma_reordering}
\end{lemma}
\begin{proof}
It is enough to observe that
$\tau(\lambda)_{\tau(B)} = \lambda_{B}$
for all subsets $B \subseteq [n]$.
\end{proof}

\section{Shellings of the Coxeter complex}
\label{section_shellings_of_Coxeter_complex}

Let $\Sigma_{n}$ be the simplicial complex formed by
the boundary of the dual of the $(n-1)$-dimensional permutahedron.
This is the type $A$ Coxeter complex.
The face poset of $\Sigma_{n}$ is the dual poset $(\Qn)^{*}$
with the order relation of the faces in $\Sigma_{n}$ by $\leq^{*}$,
the dual order.
The permutations in $\Sym_{n}$ correspond to
the facets of $\Sigma_{n}$.
We begin by showing that $\Sigma_{n}$ is a decomposable complex.

Recall that the descent set 
for a permutation $\tau \in \Sym_{n}$
is the set $\Des(\tau) = \{d_{1} < d_{2} < \cdots < d_{k}\}$
such that
$$
\tau_{1} < \tau_{2} < \cdots < \tau_{d_{1}}
>
\tau_{d_{1}+1} < \tau_{d_{1}+2} < \cdots < \tau_{d_{2}}
>
\tau_{d_{2}+1}
< 
\cdots 
<
\tau_{d_{k}}
>
\tau_{d_{k}+1} < \tau_{d_{k}+2} < \cdots < \tau_{n} .
$$
Define the ordered partition $R(\tau)$ to be
$$
R(\tau)
= 
(\{\tau_{1}, \tau_{2}, \ldots, \tau_{d_{1}}\},
\{\tau_{d_{1}+1}, \tau_{d_{1}+2}, \ldots, \tau_{d_{2}}\},
\ldots,
\{\tau_{d_{k}+1}, \tau_{d_{k}+2}, \ldots, \tau_{n}\}) . 
$$
Note that the blocks of $R(\tau)$ consist of the maximal ascending runs
in the permutation $\tau$.

Define the map $f : \Sigma_{n} \longrightarrow \Sym_{n}$
by taking an ordered partition,
ordering the elements in each block in increasing order and then
recording the elements as a permutation by reading from left to right.
Observe that, for every ordered partition $\sigma$, the permutation
$f(\sigma)$ is the minimal permutation in the weak Bruhat order
such that $\sigma\leq^* f(\sigma)$, and it is also the permutation of
minimal length such that $\sigma\leq^* f(\sigma)$.

The next two results are due to Bj\"orner;
see~\cite[Theorem~2.1]{Bjorner_II}.

\begin{proposition}[Bj\"orner]
The simplicial complex $\Sigma_{n}$
is decomposable, that is, it is
given by the disjoint union
\begin{equation}
     \Sigma_{n} = \bigcup_{\tau \in \Sym_{n}} [R(\tau), \tau] .
\label{equation_disjoint_union_of_weighted_complex}
\end{equation}
\label{proposition_Sigma_decomposition}
\end{proposition}
\begin{proof}
Define the map $r : \Sigma_{n} \longrightarrow \Sigma_{n}$
by iterating the following procedure.
If $C_{i}$ and $C_{i+1}$ are two adjacent blocks in the ordered partition $\sigma$
such that $\max(C_{i}) < \min(C_{i+1})$ then
merge these two blocks together.
It is clear that we have the inequalities
$r(\sigma) \leq^{*} \sigma \leq^{*} f(\sigma)$.
Also we have that $f(r(\sigma)) = f(\sigma)$
and $r(f(\sigma)) = r(\sigma)$.
Furthermore for a permutation (facet) $\tau$ we have that
$R(\tau) = r(\tau)$.
Hence each ordered partition
$\sigma$ appears only in the interval $[r(f(\sigma)), f(\sigma)]$.
Thus each ordered partition $\sigma$ appears
exactly once 
in the union~\eqref{equation_disjoint_union_of_weighted_complex},
and hence this union is disjoint.
\end{proof}

\begin{theorem}[Bj\"orner]
Any linear extension of the weak Bruhat order is
a shelling order of the simplicial complex $\Sigma_{n}$.
\label{theorem_linear_extension_of_Bruhat_order}
\end{theorem}
\begin{proof}
Fix a linear extension of the weak Bruhat order.
Denote it by $<_{e}$.
Consider a permutation~$\tau \in \Sym_{n}$.
Let $\sigma$ be an ordered partition
such that $\sigma \leq^{*} \tau$,
that is, $\sigma$ is a face of the facet~$\tau$.
Now consider the permutation $f(\sigma)$.
Note that the permutation $f(\sigma)$ is obtained from
the permutation~$\tau$ by merging elements together
to form blocks
and then sorting the elements in each block in increasing order.
Hence in the weak Bruhat order we have
$f(\sigma) \leq \tau$.
Hence $f(\sigma)$ appears earlier in the linear extension than $\tau$,
that is, $f(\sigma) \leq_{e} \tau$.
The result follows by
Lemma~\ref{lemma_shelling_condition}.
\end{proof}

We now consider our complex $\Sigma(\lambda)$.
\begin{theorem}
Let $\lambda = (\lambda_{1}, \ldots, \lambda_{n})$ be a sequence of
$n$ real numbers
such that $\lambdasumpositive$.
Then the complex $\Sigma(\lambda)$ is shellable.
Furthermore,
the complex $\Sigma(\lambda)$ is homeomorphic to a sphere or a ball
according to
$$
\Sigma(\lambda)
\cong
\begin{cases}
\Sss^{n-2} & \text{ if } \lambda_{1}, \lambda_{2}, \ldots, \lambda_{n} > 0, \\
\Bbbb^{n-2} & \text{ otherwise. }
\end{cases}
$$
\label{theorem_lots_of_shellings}
\end{theorem}

The proof of Theorem~\ref{theorem_lots_of_shellings}
will appear directly
after the proof of Proposition~\ref{proposition_lots_of_shellings}.

\begin{proposition}
Let $\lambda \in \Rrr^{n}$ be a sequence such that
$\lambdaweaklydecreasing$ and
$\lambdasumpositive$.
Then the elements $\mathcal{A}(\lambda)$
form a lower order ideal with respect to the weak Bruhat order
on the symmetric group $\Sym_{n}$.
\label{proposition_A_lambda}
\end{proposition}
\begin{proof}
Assume that
we have the following cover relation in the
weak Bruhat order:
$$
\tau
= 
\tau_{1} \cdots \tau_{j} \tau_{j+1} \cdots \tau_{n}
\coveredby
\tau_{1} \cdots \tau_{j+1} \tau_{j} \cdots \tau_{n}=\tau^{\prime}, $$
where 
$\tau_{j} < \tau_{j+1}$, and suppose that $\tau^{\prime}\in\mathcal{A}(\lambda)$.
To verify that $\tau$ belongs to $\mathcal{A}(\lambda)$,
it is enough to verify that
$\sum_{i=1}^{j} \lambda_{\tau_{i}}$ is nonnegative.
Note that
$\tau_{j} < \tau_{j+1}$
implies
$\lambda_{\tau_{j}} \geq \lambda_{\tau_{j+1}}$.
So
$\sum_{i=1}^{j} \lambda_{\tau_{i}}
\geq
\lambda_{\tau_{j+1}}
+
\sum_{i=1}^{j-1} \lambda_{\tau_{i}}$
which is positive by our assumption.
Hence
$\mathcal{A}(\lambda)$ is closed under the cover relation and
hence it is a lower order ideal in the weak Bruhat order.
\end{proof}

\begin{remark}
{\rm
A similar proof also shows that $\mathcal{A}(\lambda)$ is a
lower order ideal with respect to the {\em strong} Bruhat order.
}
\end{remark}

\begin{proposition}
Let $\lambda \in \Rrr^{n}$ be a sequence such that
$\lambdaweaklydecreasing$.
Then the weighted complex $\Sigma(\lambda)$
has the decomposition
$$ \Sigma(\lambda) = \bigcup_{\tau \in \mathcal{A}(\lambda)} [R(\tau), \tau] . $$
\label{proposition_Sigma_lambda_decomposition}
\end{proposition}
\begin{proof}
By Proposition~\ref{proposition_A_lambda}
the intersection $\Sigma(\lambda) \cap [R(\tau),\tau]$
is empty if $\tau$ is not in $\mathcal{A}(\lambda)$
and otherwise is the entire interval $[R(\tau), \tau]$.
Hence the result follows from
Proposition~\ref{proposition_Sigma_decomposition}.
\end{proof}

\begin{proposition}
Let $\lambda \in \Rrr^{n}$ be a sequence such that
$\lambdaweaklydecreasing$ and 
$\lambdasumpositive$.
Then the weighted complex~$\Sigma(\lambda)$
is a partial shelling of $\Sigma_{n}$
and hence~$\Sigma(\lambda)$ is shellable.
Furthermore,
the complex~$\Sigma(\lambda)$ is homeomorphic to a sphere or a ball
according to
$$
\Sigma(\lambda)
\cong
\begin{cases}
\Sss^{n-2} & \text{ if } \lambda_{n} > 0, \\
\Bbbb^{n-2} & \text{ if } \lambda_{n} \leq 0.
\end{cases}
$$
\label{proposition_lots_of_shellings}
\end{proposition}
\begin{proof}
The first statement follows by combining
Theorem~\ref{theorem_linear_extension_of_Bruhat_order}
and
Propositions~\ref{proposition_A_lambda}
and~\ref{proposition_Sigma_lambda_decomposition}.
The second statement follows
from the fact that $\Sigma_{n}$ is a sphere.
The only homology facet of the complex~$\Sigma_{n}$
is the permutation $n \cdots 21$
and it only appears as a facet in
$\Sigma(\lambda)$ if $\lambda_{n} > 0$.
\end{proof}

The proof of Theorem~\ref{theorem_lots_of_shellings}
now follows directly by combining
Lemma~\ref{lemma_reordering}
and
Proposition~\ref{proposition_lots_of_shellings}.
By taking the reduced Euler characteristic
of the simplicial complex $\Sigma(\lambda)$ in
Theorem~\ref{theorem_lots_of_shellings},
we obtain the following result 
that appears in~\cite[Corollary~A.3]{Morel} :

\begin{corollary}
Let $\lambda \in \Rrr^{n}$.
Then
$$
   \sum_{\sigma \in \mathcal{P}(\lambda)} (-1)^{|\sigma|} =
   \begin{cases}
   (-1)^{n}   & \text{ if  } \lambda_{1}, \lambda_{2}, \ldots, \lambda_{n} > 0,  \\
   0        & \text{ otherwise. }
   \end{cases}
$$
\end{corollary}

This corollary can also be proven by a sign reversing involution
on the set $\mathcal{P}(\lambda)$. 
Furthermore, such a sign reversing involution could
also be made into a discrete Morse matching;
see~\cite{Forman}.
However, the discrete Morse matching would only
give the weaker topologicial result that $\Sigma(\lambda)$
is either contractible when $\lambda_{i} \leq 0$
for some index $i$
or
homotopy equivalent to a sphere when
$\lambda_{1}, \lambda_{2}, \ldots, \lambda_{n} > 0$.

\section{The lexicographic shelling}
\label{section_lexicographic_shelling}

In this section we give a second proof that the weighted complex
$\Sigma(\lambda)$ is shellable by viewing it as the order
complex of an $EL$-shellable poset.
Throughout in this section we assume that the entries of $\lambda$
are pairwise distinct. We can always perturb $\lambda$
slightly to make this condition true. This does not
change the complex $\Sigma(\lambda)$.
We also assume that $\lambdasumpositive$.

Let $P$ be a graded poset  that 
has a minimal element $\hz$ and a maximal element $\ho$
and where all maximal chains have the same length.
Let $\mathcal{C}(P)$ be the set of all cover relations of $P$,
that is,
$$ \mathcal{C}(P) = \{(x,y) \in P^{2} \:\: : \:\: x \coveredby y\} . $$
A {\em labeling} of a poset $P$ is a map
$\kappa$ from the set of cover relations $\mathcal{C}(P)$
to a linearly ordered set.
For a saturated chain
$c = \{x_{0} \coveredby x_{1} \coveredby \cdots \coveredby x_{k}\}$
let $\kappa(c)$
denote the word of labels, that is,
$\kappa(c) = 
\kappa(x_{0},x_{1}) \kappa(x_{1},x_{2}) \cdots \kappa(x_{k-1},x_{k})$.
For a graded poset
$P$,
an {\em $R$-labeling} is a labeling~$\kappa$
such that in each interval $[x,y]$ there is a unique chain
$c = \{x = x_{0} \coveredby x_{1} \coveredby \cdots \coveredby x_{k} = y\}$
where the word $\kappa(c)$ is rising, that is,
$\kappa(x_{0},x_{1}) \leq \kappa(x_{1},x_{2}) \leq \cdots \leq
\kappa(x_{k-1},x_{k})$.
An {\em $EL$-labeling} is an $R$-labeling with the extra condition
that the word $\kappa(c)$
corresponding to the
unique rising chain~$c$ in the interval $[x,y]$
is also lexicographically least
among the set of
all the words arising from saturated chains in the interval.
It is
well-known that
a poset
having an $EL$-labeling
implies that the order complex $\Delta(P - \{\hz,\ho\})$
is a shellable simplicial complex; see~\cite[Section~2]{Bjorner}.

Let $B_{n}$ denote the Boolean algebra,
that is, the collection of subsets of the set $[n]$
ordered by inclusion.
Define the subposet $B(\lambda)$
by 
$$ B(\lambda) = \{ S \in B_{n} \:\: : \:\: \lambda_{S} > 0\} \cup \{\emptyset\}. $$
Note that the full set $[n]$
belongs to $B(\lambda)$ and it is the maximal element of $B(\lambda)$.

We view this definition geometrically as follows.
Identify the vertices of the $n$-dimensional cube, that is,
$\{0,1\}^{n}$ with the Boolean algebra $B_{n}$.
The subposet $B(\lambda)$ then consists of those
vertices $(x_{1}, x_{2}, \ldots, x_{n})$
that satisfy the linear inequality
$\lambda_{1} x_{1} + \lambda_{2} x_{2} + \cdots + \lambda_{n} x_{n} > 0$
as well as the zero vector.
\begin{lemma}
Let $S$ and $T$ be two subsets in the poset $B(\lambda)$
such that $T \subsetneq S$. Let $a$ be the element
in the set difference $S-T$ with the maximal $\lambda$-value.
Then the set $T \cup \{a\}$ also belongs to $B(\lambda)$.
\label{lemma_one_step}
\end{lemma}
\begin{proof}
The proof is the same as the proof of
Lemma~\ref{lemma_split}.
If $\lambda_{a} \geq 0$ then
$\lambda_{T \cup \{a\}}
=
\lambda_{T}
+
\lambda_{a}
> 0$.
If $\lambda_{a} \leq 0$ then
all the elements in $S-T$ have non-positive $\lambda$-values,
hence
$\lambda_{T \cup \{a\}} \geq \lambda_{S} > 0$.
\end{proof}

From this straightforward observation we have
the following consequences.
\begin{proposition}
The poset $B(\lambda)$ is graded of rank $n$.
Furthermore, by labeling the cover relation
$T \subsetneq S$, where $|T|+1=|S|$, 
by $-\lambda_{a}$, where $a$ is the unique element in the set difference
$S-T$, we obtain an $EL$-labeling of the poset $B(\lambda)$.
\end{proposition}
\begin{proof}
By repeatedly applying Lemma~\ref{lemma_one_step}
we obtain that every interval $[T,S]$ has
a saturated chain consisting of $|S|-|T|$ steps.
Thus the poset $B(\lambda)$ has cardinality
as a rank function.
Observe that the labels of any saturated chain in the interval $[T,S]$
are the negatives
of the $\lambda$-values 
of a permutation of elements in the set difference $S-T$.
Also, the chain obtained
by repeatedly applying Lemma~\ref{lemma_one_step}
has its labels
in increasing order.
It is the only such increasing chain. Furthermore,
it is also the lexicographically least chain.
Hence the labeling is an $EL$-labeling.
\end{proof}

Finally observe that
the order complex of $B(\lambda)$,
that is,
$\Delta(B(\lambda) - \{\hz,\ho\})$,
is the complex $\Sigma(\lambda)$.
Thus we obtain that 
we can shell the facets of $\Sigma(\lambda)$
in lexicographic order.

\section{The main result}
\label{section_main_result}

Define the map
$g : \Sigma(\lambda) \longrightarrow \Sym_{n}$
by taking  an ordered partition $\sigma$,
ordering the elements in each block in decreasing order,
and then
recording the elements as a permutation by reading from left to right.
This is similar to the map $f$ defined before
Proposition~\ref{proposition_Sigma_decomposition}.
The signs of the two permutations $f(\sigma)$ and $g(\sigma)$
are related by
$$
 (-1)^{g(\sigma)} = (-1)^{\sum_{i=1}^{k} \binom{c_{i}}{2}} \cdot (-1)^{f(\sigma)} ,
$$
where $\sigma = (C_{1},C_{2}, \ldots, C_{k})$ and $c_{i} = |C_{i}|$.

Define $S(\lambda)$ to be the sum
\begin{equation}
S(\lambda)
=
\sum_{\sigma \in \Sigma(\lambda)}
(-1)^{|\sigma|} \cdot (-1)^{g(\sigma)} . 
\label{equation_S_lambda}
\end{equation}

Let $M_{n}$ be the set of all
maximal matchings
on the set $\{1,2, \ldots,n\}$.
We say that two edges $\{a,c\}$ and $\{b,d\}$ of a matching $p$
\emph{cross} if $a<b<c<d$. Let $\cross(p)$ denote the number of crossings.
Define the sign of a matching $p$ in $M_{n}$
by
$$   (-1)^{p} = (-1)^{\cross(p)} \cdot
\begin{cases}
1 & \text{ if $n$ is even,} \\
(-1)^{i-1} & \text{ if $n$ is odd,}
\end{cases}
$$
where $i$ is the unique isolated vertex if $n$ is odd.

\begin{remark}
{\rm
When $n$ is odd
there is a sign-preserving bijection
between $M_{n}$ and $M_{n+1}$
by joining the isolated vertex $i$ to the new vertex $n+1$.
}
\label{remark_add_extra_edge}
\end{remark}

Define two functions $c_{1} : \Rrr \fl \Rrr$ and
$c_{2} : \Rrr^{2} \fl \Rrr$ in the following manner:
\[c_{1}=\ungras_{\Rrr_{>0}},
\:\:\:\: \text{ and } \:\:\:\:
c_{2}(a,b)=\left\{\begin{array}{ll}1 & \mbox{ if }a,b>0, \\
2 & \mbox{ if } a > -b \geq 0, \\
0 & \mbox{ otherwise.}\end{array}\right.\]
If $e=\{i,j\}$ is an edge with $i<j$, we set
$c(e,\lambda)=c_{2}(\lambda_{i},\lambda_{j})$.

Let $m$ be the number of edges of a maximal matching,
that is, $m = \lfloor n/2 \rfloor$.
Let $p$ be a maximal matching in $M_{n}$.
If $n$ is even, we write $p$ as the collection of
edges  $\{e_{1}, e_{2}, \ldots, e_{m}\}$.
If $n$ is odd, the matching $p$
is of the form
$\{e_{1}, e_{2}, \ldots, e_{m}\}$
together with an isolated vertex $i$.
We set
$$
c(p,\lambda)=\prod_{j=1}^{m} c(e_{j},\lambda) 
\cdot
\begin{cases}
1 & \text{if $n$ is even,} \\
c_{1}(\lambda_{i}) & \text{if $n$ is odd.}
\end{cases}
$$

Finally, let
\begin{equation}
     T(\lambda)=\sum_{p\in M_{n}} (-1)^{p} \cdot c(p,\lambda) . 
\label{equation_T_lambda}
\end{equation}

\begin{remark}
{\rm
Following~\cite{Herb}
let $\Sym_{n}^{**}$
be the set of all permutations $\tau$
in the symmetric group $\Sym_{n}$
satisfying the inequalities
$\tau_{2j-1} < \tau_{2j}$ for
$1 \leq j \leq m$
and
$\tau_{1} < \tau_{3} < \cdots < \tau_{2m-1}$.
Then there is a bijection
$\Sym_{n}^{**} \iso M_{n}$
by sending $\tau$ to the matching
$p=\{\{\tau_{1},\tau_{2}\}, \{\tau_{3},\tau_{4}\},\ldots,
\{\tau_{2m-1},\tau_{2m}\}\}$.
Note that when $n$ is odd, the isolated vertex is $\tau_{n}$.
This bijection preserves the sign, that is,
$(-1)^{\tau} = (-1)^{p}$. 
}
\label{remark_Sn**}
\end{remark}

\begin{remark}
{\rm
When $n$ is even we can express the
sum $T(\lambda)$ as the Pfaffian of a skew-symmetric matrix.
Let $A$ be the skew-symmetric matrix of order $n$
with the upper triangular entries given by
$A_{i,j}=c_2(\lambda_{i},\lambda_{j})$.
The Pfaffian of $A$ is then the sum $T(\lambda)$. (See Remark \ref{remark_Sn**}.)

Similarly, when $n$ is odd, we use the bijection in
Remark~\ref{remark_add_extra_edge}.
In this case let $A$ be the skew-symmetric
matrix of order $n+1$ where
the upper triangular entries are 
$$ A_{i,j}
=
\begin{cases}
c_{2}(\lambda_{i},\lambda_{j})
& \text{ if } 1 \leq i < j \leq n , \\
c_{1}(\lambda_{i})
& \text{ if } 1 \leq i \leq n, j=n+1 .
\end{cases}
$$
Again the Pfaffian of $A$ is the sum $T(\lambda)$.
}
\end{remark}

We can now state the main result.
This
is Proposition~A.4 of~\cite{Morel}.
\begin{theorem}
For every $\lambda \in \Rrr^{n}$, we have
\[S(\lambda)=(-1)^{n} \cdot T(\lambda) . \]
\label{theorem_S_equal_T}
\end{theorem}
The proof of Theorem~\ref{theorem_S_equal_T} will
be given
at the end of Section~\ref{section_permute}.

\section{The base case}
\label{section_base_case}

Let $\overline{\lambda}$ denote the
reverse of $\lambda$, that is,
$\overline{\lambda} = (\lambda_{n}, \ldots, \lambda_{2}, \lambda_{1})$.
\begin{lemma}
The sum $S(\lambda)$ can be expressed as
$$
\sum_{\sigma \in \Sigma(\lambda)}
(-1)^{|\sigma|} \cdot (-1)^{g(\sigma)} 
=
(-1)^{\binom{n}{2}}
\cdot
\sum_{\sigma \in \Sigma(\overline{\lambda})}
(-1)^{|\sigma|} \cdot (-1)^{f(\sigma)} . 
$$
\label{lemma_reverse}
\end{lemma}
\begin{proof}
Let $\tau_{0}$ denote the longest permutation
in $\Sym_{n}$,
that is,
$\tau_{0} = n \cdots 21$.
Note that $(-1)^{\tau_{0}} = (-1)^{\binom{n}{2}}$
and $\overline{\lambda} = \tau_{0} \lambda$.
For any ordered partition $\sigma \in \Qn$ 
we have the
following equality between permutations:
$\tau_{0} g(\sigma) = f(\tau_{0} \sigma)$.
In particular, their signs agree, that is,
$(-1)^{\binom{n}{2}} \cdot (-1)^{g(\sigma)}
= 
(-1)^{f(\tau_{0} \sigma)}$.
Now apply Lemma~\ref{lemma_reordering}
using the permutation $\tau_{0}$.
The result follows after summing over all faces
$\sigma \in \Sigma(\lambda)$.
\end{proof}

\begin{proposition}
Let $\lambda \in \Rrr^{n}$ be a sequence such that
$\lambdaweaklyincreasing$.
Then the sum $S(\lambda)$ is given by
$$
S(\lambda)
=
\begin{cases}
(-1)^{n} & \text{ if } \lambda_{1} > 0, \\
0 & \text{ otherwise.}
\end{cases}
$$ 
\label{proposition_S_lambda_increasing}
\end{proposition}
\begin{proof}
Lemma~\ref{lemma_reverse}
asserts that
$$
S(\lambda)
=
(-1)^{\binom{n}{2}}
\cdot
\sum_{\sigma\in\Sigma(\overline{\lambda})}
(-1)^{|\sigma|} \cdot (-1)^{f(\sigma)}.
$$
Moreover,
the entries of 
$\overline{\lambda}$
are in weakly decreasing order.
Hence we can use
the decomposition
in Proposition~\ref{proposition_Sigma_lambda_decomposition}
for the complex 
$\Sigma(\overline{\lambda})$,
that is, we have
$$
\Sigma(\overline{\lambda})
=
\bigcup_{\tau \in \mathcal{A}(\overline{\lambda})} [R(\tau), \tau] .
$$
Observe the function
$(-1)^{f(\sigma)}$
is constant on the interval
$[R(\tau), \tau]$.
Hence the sum
$$ \sum_{\sigma \in [R(\tau), \tau]} (-1)^{|\sigma|} \cdot (-1)^{f(\sigma)} $$
is zero unless $R(\tau) = \tau$.
But we can only have $R(\tau)=\tau$
when $\tau$ is the longest permutation
$n \cdots 21$. For this permutation
the sum is $(-1)^{n} \cdot (-1)^{\binom{n}{2}}$.
This permutation occurs in $\mathcal{A}(\overline{\lambda})$
if and only if $\lambda_{1} > 0$.
Hence the result follows.
\end{proof}

\begin{lemma}
The following two evaluations of $T(\lambda)$ hold:
\begin{itemize}
\item[(1)] If $\lambda_{1},\lambda_{2},\ldots,\lambda_{n}>0$, then $T(\lambda)=1$.
\item[(2)] If $\lambda_{1}\leq 0$, then $T(\lambda)=0$.
\end{itemize}
\label{lemma_T_two_statements}
\end{lemma}
\begin{proof}
We begin by proving the first statement.
By the hypothesis on $\lambda$, for every maximal matching~$p$,
we have $c(p,\lambda)=1$.
We prove 
$T(\lambda) = 1$ for all even $n$ by induction.
The induction basis $n=2$ is straightforward to check
since there is only one maximal matching on two vertices.
Assume now it is true for $n-2$.
Observe that the parity of the number 
of crossings to which the edge $\{1,j\}$ contributes is
the same as the parity of~$j$. 
Let $q \in M_{n-2}$ be the matching obtained from $p$ by removing
the vertices $1$ and $j$, and then applying 
the unique order-preserving relabeling of the remaining vertices.
This gives a bijection from the set of maximal matchings $p\in M_{n}$ such that
$1$ and $j$ are matched and the set of maximal matchings $q\in M_{n-2}$, and we
have $(-1)^p=(-1)^j (-1)^q$.
Hence the sum is given by
$$
T(\lambda)
=
\sum_{p \in M_{n}} (-1)^{p}
=
\sum_{j=2}^{n}  (-1)^{j} \cdot \sum_{q \in M_{n-2}} (-1)^{q}
=
\sum_{j=2}^{n}  (-1)^{j} 
=
1 ,
$$
where the third step is the induction hypothesis.

Assume now that $n$ is odd.
By using the sign-preserving bijection in
Remark~\ref{remark_add_extra_edge}
between $M_{n}$ and~$M_{n+1}$,
we obtain that
the two sums
$\sum_{p \in M_{n}} (-1)^{p}$
and
$\sum_{r \in M_{n+1}} (-1)^{r}$
are equal, and we have already seen that the second sum is equal to $1$.

We next prove the second statement.
Suppose that $\lambda_{1} \leq 0$. 
Then for every maximal matching~$p$, we claim $c(p,\lambda)=0$. 
Indeed, if the vertex $1$ is isolated then
$c_{1}(\lambda_{1}) = 0$ is a factor of $c(p, \lambda)$.
If not, let $e = \{1,j\}$ be the edge of $p$ containing the vertex $1$.
The contribution of $e$ to $c(p,\lambda)$ is $c_{2}(\lambda_{1},\lambda_{j})=0$
which is independent of the value of $\lambda_{j}$.
In both cases, $c(p,\lambda)$ has a factor equal to~$0$, so it is~$0$. 
\end{proof}

By combining
Proposition~\ref{proposition_S_lambda_increasing}
and
Lemma~\ref{lemma_T_two_statements},
we obtain the following corollary.
\begin{corollary} 
If $\lambda_{1} \leq \lambda_{2} \leq \cdots \leq \lambda_{n}$, then
$S(\lambda)=(-1)^{n} \cdot  T(\lambda)$.
\label{corollary_base_case}
\end{corollary}

\section{Permuting the entries of $\lambda$}
\label{section_permute}

We now give a way to reduce the general case to the case treated
in the previous section,
assuming that we know the identity for smaller values of $n$.
This reduction is
directly inspired by Herb's paper~\cite{Herb} 
on discrete series characters. 

Let $\lambda \in \Rrr^{n}$
with no ordering hypothesis on $\lambda$.
Suppose that $n \geq 2$, and fix $1 \leq i \leq n-1$.
Recall that $s_{i}$ is the simple transposition $(i,i+1)$
in the symmetric group $\Sym_{n}$.
Also, we write
$\mu=(\lambda_{1},\ldots,\lambda_{i-1},\lambda_{i+2},\ldots,\lambda_{n})\in\R^{n-2}$.

\begin{proposition}
The following two identities hold:
\begin{align}
S(\lambda) + S(s_{i}\lambda)
& =
-2 \cdot \ungras_{\lambda_{i}+\lambda_{i+1}>0}\cdot S(\mu) ,
\label{equation_S} \\
T(\lambda) + T(s_{i}\lambda)
& =
2 \cdot \ungras_{\lambda_{i}+\lambda_{i+1}>0}\cdot T(\mu) .
\label{equation_T}
\end{align}
\label{proposition_S_equal_T}
\end{proposition}
\begin{proof}
We begin by proving~\eqref{equation_S}.
For $\sigma \in \Sigma_{n}$, let $s_{i}\sigma$ denote the
ordered partition where we exchange the elements $i$ and $i+1$.
Note that $s_{i}$ is an involution on $\Sigma_{n}$
and that it preserves the number of blocks.
We write
$\Sigma(\lambda) =\Sigma^{\prime}(\lambda) \sqcup \Sigma^{\prime\prime}(\lambda)$,
where $\Sigma^{\prime\prime}(\lambda)$ is the set of fixed points of $s_{i}$
in~$\Sigma(\lambda)$,
that is, 
$\Sigma^{\prime}(\lambda)$ is the set of $\sigma\in\Sigma(\lambda)$
such that $i$ and $i+1$ are in different blocks of $\sigma$
and
$\Sigma^{\prime\prime}(\lambda)$ is the set of $\sigma\in\Sigma(\lambda)$
such that $i$ and $i+1$ are in the same block of $\sigma$.

Note that the action of $s_{i}$ gives a bijection between
$\Sigma^{\prime}(\lambda)$ and $\Sigma^{\prime}(s_{i} \lambda)$.
Furthermore, for an ordered partition $\sigma$ in $\Sigma^{\prime}(\lambda)$
we have
$g(s_{i}\sigma) = s_{i}\cdot g(\sigma)$.
We obtain
\begin{align}
\sum_{\sigma \in \Sigma^{\prime}(\lambda)}
(-1)^{|\sigma|} \cdot (-1)^{g(\sigma)}
+
\sum_{\sigma \in \Sigma^{\prime}(s_{i}(\lambda))}
(-1)^{|\sigma|} \cdot (-1)^{g(\sigma)}
& = 0 
.
\label{equation_Sigma_prime}
\end{align}
We next consider $\Sigma^{\prime\prime}(\lambda)$.
Note that $\Sigma^{\prime\prime}(\lambda) =\Sigma^{\prime\prime}(s_{i}\lambda)$, so 
it remains to show that
\begin{align}
\sum_{\sigma\in\Sigma^{\prime\prime}(\lambda)}
(-1)^{|\sigma|} \cdot (-1)^{g(\sigma)}
& =
-\ungras_{\lambda_{i}+\lambda_{i+1}>0}\cdot S(\mu).
\label{equation_Sigma_prime_prime}
\end{align}

Suppose first that $\lambda_{i}+\lambda_{i+1}\leq 0$.
Define a matching on $\Sigma^{\prime\prime}(\lambda)$ in the following manner.
If $\sigma=(C_{1},C_{2},\ldots,C_{r}) \in \Sigma^{\prime\prime}(\lambda)$
satisfies
$i, i+1 \in C_{s}$ and $|C_{s}|\geq 3$, match it with
$$ \sigma^{\prime}
=
(C_{1},\ldots,C_{s-1},C_{s}-\{i,i+1\},\{i,i+1\},C_{s+1},\ldots,C_{r}) . $$
The hypothesis that $\lambda_{i}+\lambda_{i+1} \leq 0$
and Lemma~\ref{lemma_split_block}
imply that 
$\sigma^{\prime} \in \Sigma^{\prime\prime}(\lambda)$.
Furthermore, this is a perfect matching
since no ordered partition 
in $\Sigma^{\prime\prime}(\lambda)$
can begin with the block $\{i,i+1\}$.
It is easy to see that 
$(-1)^{g(\sigma)}=(-1)^{g(\sigma^{\prime})}$ if $\sigma$ and $\sigma^{\prime}$ are matched.
So we obtain
\begin{align*}
\sum_{\sigma\in\Sigma^{\prime\prime}(\lambda)}
(-1)^{|\sigma|} \cdot (-1)^{g(\sigma)}
&=0.
\end{align*}

Now assume that $\lambda_{i}+\lambda_{i+1}>0$.
We consider another matching on $\Sigma^{\prime\prime}(\lambda)$, defined as follows.
If $\sigma=(C_{1}, C_{2}, \ldots, C_{r})\in\Sigma^{\prime\prime}(\lambda)$
satisfies $i, i+1 \in C_{s}$ and $|C_{s}| \geq 3$, match it with
$$ \sigma^{\prime}
= 
(C_{1},\ldots,C_{s-1},\{i,i+1\},C_{s}-\{i,i+1\},C_{s+1},\ldots,C_{r}). $$
The hypothesis that $\lambda_{i}+\lambda_{i+1} > 0$
and Lemma~\ref{lemma_split_block}
imply that 
$\sigma^{\prime} \in \Sigma^{\prime\prime}(\lambda)$.
The unmatched elements
are the ordered partitions whose last block is $\{i, i+1\}$.
Again, it is straightforward to see that 
$(-1)^{g(\sigma)}=(-1)^{g(\sigma^{\prime})}$ if $\sigma$ and $\sigma^{\prime}$ are matched.
Denote by $\Sigma^{\prime\prime\prime}(\lambda)$ the set of unmatched elements
in~$\Sigma^{\prime\prime}(\lambda)$.
We obtain
$$ \sum_{\sigma\in\Sigma^{\prime\prime}(\lambda)}
(-1)^{|\sigma|} \cdot (-1)^{g(\sigma)}
=
\sum_{\sigma\in\Sigma^{\prime\prime\prime}(\lambda)}
(-1)^{|\sigma|} \cdot (-1)^{g(\sigma)} . 
$$
Identify
$[n-2]$ and $[n]-\{i,i+1\}$
using the unique order-preserving bijection between these sets.
We note that
$\sigma=(C_{1}, \ldots, C_{r},\{i,i+1\})\in\Sigma_{n}$
is an element of $\Sigma(\lambda)$ if and only if
$\tau := (C_{1},\ldots,C_{r})$ is an element of $\Sigma(\mu)$.
This induces a bijection $\Sigma^{\prime\prime\prime}(\lambda)\iso\Sigma(\mu)$.
Also, we have $|\sigma|=|\tau|+1$ and $(-1)^{g(\sigma)}=(-1)^{g(\tau)}$.
So finally we find that
\begin{align*}
\sum_{\sigma\in\Sigma^{\prime\prime\prime}(\lambda)}
(-1)^{|\sigma|} \cdot (-1)^{g(\sigma)}
&=
-
\sum_{\tau\in\Sigma(\mu)}
(-1)^{|\tau|} \cdot (-1)^{g(\tau)}
=
-S(\mu).
\end{align*}
Identity~\eqref{equation_S}
follows by combining equation~\eqref{equation_Sigma_prime} 
with twice equation~\eqref{equation_Sigma_prime_prime}
in the two cases
$\Sigma^{\prime\prime}(\lambda)$ and $\Sigma^{\prime\prime}(s_{i} \lambda)$.

Identity~\eqref{equation_T} 
is proved as in Case~I of Lemma~2.14 in~\cite{Herb}.
Let $M_{n}^{\prime\prime}$ be the set
of maximal matchings in $M_{n}$ where
$\{i,i+1\}$ is an edge and
let $M_{n}^{\prime}$ be the complement,
that is, where the vertices $i$ and $i+1$ are not matched with each other.

For $p \in M_{n}$, let $s_{i}p$ denote the
matching where we exchange the vertices $i$ and $i+1$.
Note that $s_{i}$ is an involution on the
set $M_{n}$
and that $M_{n}^{\prime\prime}$ is the set of fixed points
of $s_{i}$.
Furthermore, for a matching $p \in M_{n}^{\prime}$
we have
$(-1)^{s_{i} p} = - (-1)^{p}$
and
$c(s_{i}p, s_{i} \lambda) = c(p, \lambda)$.
Hence we obtain
\begin{align}
\sum_{p \in M_{n}^{\prime}} (-1)^{p} \cdot c(p, \lambda)
+
\sum_{p \in M_{n}^{\prime}} (-1)^{p} \cdot c(p, s_{i} \lambda)
= 0 .
\label{equation_M_prime}
\end{align}
Again we use
the unique order-preserving bijection $[n-2] \iso [n] - \{i,i+1\}$.
Thus if $p$ is a matching in $M_{n}^{\prime\prime}$, we can view
$p-\{i,i+1\}$ as a matching
on the vertex set $[n-2]$, which we will denote by~$q$.
This gives a bijection between
$M_{n}^{\prime\prime}$
and $M_{n-2}$.
Note that the two matchings $p$ and $q$ have the same sign,
that is, $(-1)^{p} = (-1)^{q}$.
Furthermore
\begin{align*}
c(p,\lambda) + c(p,s_{i}\lambda) 
& =
(c_{2}(\lambda_{i},\lambda_{i+1}) + c_{2}(\lambda_{i+1},\lambda_{i}))
\cdot
c(q,\mu) \\
& =
2 \cdot \ungras_{\lambda_{i}+\lambda_{i+1} > 0}
\cdot
c(q,\mu) .
\end{align*}
Now summing over all $p \in M_{n}^{\prime\prime}$
we obtain
\begin{align}
\sum_{p \in M_{n}^{\prime\prime}}
(-1)^{p} \cdot c(p, \lambda)
+
\sum_{p \in M_{n}^{\prime\prime}}
(-1)^{p} \cdot c(p, s_{i} \lambda)
& = 
2 \cdot \ungras_{\lambda_{i}+\lambda_{i+1} > 0}
\cdot
\sum_{q \in M_{n-2}} (-1)^{q} \cdot c(q, \mu) .
\label{equation_M_prime_prime}
\end{align}
Identity~\eqref{equation_T} 
now follows by summing equations~\eqref{equation_M_prime}
and~\eqref{equation_M_prime_prime}.
\end{proof}

We now include the proof of Theorem~\ref{theorem_S_equal_T}. 

\begin{proof}[Proof of Theorem~\ref{theorem_S_equal_T}.]
We proceed by induction on $n \geq 1$.
The induction basis is $n \leq 2$, and is straightforward to verify.
Assume the theorem is true for $n-2 \geq 1$ and let us prove it for~$n$.
Proposition~\ref{proposition_S_equal_T} and the induction hypothesis imply that for
every $\lambda \in \Rrr^{n}$ and $i \in [n-1]$ 
the theorem is true for $\lambda$
if and only if it is true for $s_{i}\lambda$.
As the transpositions $s_{1}, s_{2}, \ldots, s_{n-1}$
generate the symmetric group $\Sym_{n}$, we deduce 
that the theorem is true for
$\lambda$ if and only if there exists 
a permutation $\tau\in\Sym_{n}$ such that the theorem is
true for $\tau\lambda$.
But for every $\lambda \in \Rrr^{n}$, we can find a $\tau \in \Sym_{n}$
such that the entries of $\tau\lambda$ are in weakly increasing order. 
The
result for $\tau\lambda$ is exactly Corollary~\ref{corollary_base_case}, and so we
finally deduce the result for our original $\lambda$,
completing the induction step.
\end{proof}

\section{An expression for decreasing $\lambda$}
\label{section_expression}

The base case of the proof of
Theorem~\ref{theorem_S_equal_T}
is when the sequence $\lambda$ is
weakly increasing. In this section
we consider the other extreme,
that is, when
the sequence $\lambda$ is weakly decreasing.
We give an expression for $S(\lambda)$
using a permutation statistic
on the facets of $\Sigma(\lambda)$,
that is, the set $\mathcal{A}(\lambda)$.

We begin by defining two sequences
$a_{i}$ and $b_{i}$ by
$a_{i}  = (-1)^{\binom{i}{2} + 1}$
for $i \geq 1$,
$b_{0} = b_{1} = 1$ and $b_{i} = 2$ for $i \geq 2$.
\begin{lemma}
The following identity holds:
$$ \sum_{\vec{c} \in \Comp(n)}
a_{c_{1}} \cdot a_{c_{2}} \cdots a_{c_{k}}
=
(-1)^{n} \cdot b_{n} , $$
where the sum is over all compositions
$\vec{c} = (c_{1}, c_{2}, \ldots, c_{k})$
of $n$.
\label{lemma_generating_function_proof}
\end{lemma}
\begin{proof}
It is enough to see that
\begin{align*}
a(x)
& =
\sum_{n \geq 1} a_{n} \cdot x^{n}
=
\frac{-x \cdot (1-x)}{1+x^{2}} \\
\intertext{and}
\frac{1}{1-a(x)}
& =
\frac{1+x^{2}}{1+x}
=
1 + \sum_{n \geq 1} (-1)^{n} \cdot b_{n} \cdot x^{n}.
\qedhere
\end{align*}
\end{proof}

\begin{proposition}
Let $\id = 12 \cdots n$ be the identity permutation.
Then the following identity holds:
$$
\sum_{\sigma \in [R(\id),\id]} (-1)^{|\sigma|} \cdot (-1)^{g(\sigma)}
=
(-1)^{n} \cdot b_{n} .
$$
\label{proposition_b_n}
\end{proposition}
\begin{proof}
The interval
$[R(\id),\id] \in \Sigma(\lambda)$
is isomorphic to the poset $\Comp(n)$.
Hence we view the elements of this interval as
compositions $\vec{c} = (c_{1}, c_{2}, \ldots, c_{k})$ of $n$
where $\sigma$ is the ordered partition
$\sigma = (C_{1},C_{2}, \ldots, C_{k})$
and
$C_{i}$ is the interval $[c_{1}+c_{2}+\cdots+c_{i-1}+1, c_{1}+c_{2}+\cdots+c_{i}]$.
Note that $|C_{i}| = c_{i}$.
Now the sum is given by
\begin{align*}
\sum_{\vec{c} \in \Comp(n)} (-1)^{k} \cdot  \prod_{i=1}^{k} (-1)^{\binom{c_{i}}{2}}
&=
\sum_{\vec{c} \in \Comp(n)} \prod_{i=1}^{k} (-1)^{\binom{c_{i}}{2} + 1}
=
\sum_{\vec{c} \in \Comp(n)} \prod_{i=1}^{k} a_{c_{i}}
=
(-1)^{n} \cdot b_{n} , 
\end{align*}
where the last equality is by Lemma~\ref{lemma_generating_function_proof}.
\end{proof}

For a permutation $\tau$ with
descent composition
$(c_{1}, c_{2}, \ldots, c_{k})$
define
$$  b(\tau) = \prod_{i=1}^{k} b_{c_{i}} .  $$
In other words, $b(\tau)$ is $2$ to the power of the number of maximal
ascent runs in $\tau$ which have size greater than or equal to $2$.
\begin{theorem}
When $\lambda\in\Rrr^n$ is such that 
$\lambdaweaklydecreasing$
the following identity holds:
$$
S(\lambda)
=
(-1)^{n}
\cdot 
\sum_{\tau \in \mathcal{A}(\lambda)} (-1)^{\tau} \cdot b(\tau)  .
$$
\end{theorem}
\begin{proof}
Let $\tau$ be a permutation in $\mathcal{A}(\lambda)$
with descent composition $\vec{c} = (c_{1}, c_{2}, \ldots, c_{k})$.
Let $\tau^{(i)}$ be the $i$th descent run of the permutation $\tau$,
that is, $\tau^{(i)}$ is a partial permutation
and
$\tau$ can be written as the concatenation
of $\tau^{(1)}$ through $\tau^{(k)}$.
We sum over all $\sigma$ in the interval
$[R(\tau),\tau]$.
In order to do so, we write $\sigma$ as a concatenation
of ordered partitions
$\sigma^{(1)}$, 
$\sigma^{(2)}$, $\ldots$,
$\sigma^{(k)}$,
where
$\sigma^{(i)}$
is an ordered composition of the set of elements of $\tau^{(i)}$.
The ordered partition $\sigma^{(i)}$ belongs to the interval
$[R(\tau^{(i)}),\tau^{(i)}]$ in the ordered partition lattice defined
on the elements of $\tau^{(i)}$.

Let $\vec{d} = (d_{1}, d_{2}, \ldots, d_{m})$ be the block sizes of $\sigma$,
that is, $\type(\sigma) = \vec{d}$.
Note that $\vec{d} \leq \vec{c}$ in the poset $\Comp(n)$.
Hence we write
$\vec{d}$
as the concatenation
$\vec{d}^{(1)} \circ \vec{d}^{(2)} \circ \cdots  \circ \vec{d}^{(k)}$
where $\type(\sigma^{(i)}) = \vec{d}^{(i)}$.

Rewrite the sum over the elements in the interval
$[R(\tau),\tau]$ as follows:
\begin{align*}
\sum_{\sigma \in [R(\tau),\tau]} (-1)^{|\sigma|} \cdot (-1)^{g(\sigma)}
& =
\sum_{\sigma \in [R(\tau),\tau]}
(-1)^{f(\sigma)} \cdot 
(-1)^{|\sigma| + \sum_{j=1}^{m} \binom{d_{j}}{2}} \\
& =
(-1)^{\tau} \cdot 
\sum_{\sigma \in [R(\tau),\tau]}
\prod_{j=1}^{m} a_{d_{j}} \\
& =
(-1)^{\tau} \cdot 
\prod_{i=1}^{k}
\sum_{\sigma^{(i)} \in [R(\tau^{(i)}),\tau^{(i)}]}
\prod_{j=1}^{|\vec{d}^{(i)}|}
a_{d^{(i)}_{j}}
.
\end{align*}
By Proposition~\ref{proposition_b_n}
the sum is $(-1)^{c_{i}} \cdot b_{c_{i}}$.
Finally, the product is
$(-1)^{n} \cdot (-1)^{\tau} \cdot b(\tau)$.
By using
the decomposition of $\Sigma(\lambda)$ in
Proposition~\ref{proposition_Sigma_lambda_decomposition}
and
summing over all permutations $\tau$ in $\mathcal{A}(\lambda)$,
the result follows.
\end{proof}

\section*{Acknowledgements}

We thank the two referees for their comments and suggestions
on the exposition.
We would also like to thank 
the \'Ecole Normale Sup\'erieure de Lyon (\'ENS de Lyon)
for its hospitality and support to the second author during
the academic year 2017--2018
and to the first and third author
during one week visits to \'ENS de Lyon.
The first and third author also 
thank the Institute for Advanced Study in Princeton 
for hosting a research visit in Summer 2018.
This work was partially supported by grants from the
Simons Foundation
(\#429370 to Richard~Ehrenborg
and \#422467 to
Margaret Readdy).

\newcommand{\journal}[6]{{\sc #1,} #2, {\it #3} {\bf #4} (#5), #6.}
\newcommand{\book}[4]{{\sc #1,} ``#2,'' #3, #4.}
\newcommand{\bookf}[5]{{\sc #1,} ``#2,'' #3, #4, #5.}
\newcommand{\books}[6]{{\sc #1,} ``#2,'' #3, #4, #5, #6.}
\newcommand{\collection}[6]{{\sc #1,}  #2, #3, in {\it #4}, #5, #6.}
\newcommand{\thesis}[4]{{\sc #1,} ``#2,'' Doctoral dissertation, #3, #4.}
\newcommand{\springer}[4]{{\sc #1,} ``#2,'' Lecture Notes in Math.,
                          Vol.\ #3, Springer-Verlag, Berlin, #4.}
\newcommand{\preprint}[3]{{\sc #1,} #2, preprint #3.}
\newcommand{\preparation}[2]{{\sc #1,} #2, in preparation.}
\newcommand{\appear}[3]{{\sc #1,} #2, to appear in {\it #3}}
\newcommand{\submitted}[3]{{\sc #1,} #2, submitted to {\it #3}}
\newcommand{\JCTA}{J.\ Combin.\ Theory Ser.\ A}
\newcommand{\AdvancesinMathematics}{Adv.\ Math.}
\newcommand{\JournalofAlgebraicCombinatorics}{J.\ Algebraic Combin.}

\newcommand{\communication}[1]{{\sc #1,} personal communication.}




\bigskip

\small

\noindent
{\sc University of Kentucky,
Department of Mathematics,
Lexington, KY 40506.} \hfill\break
{\tt richard.ehrenborg@uky.edu}.

\vspace{2mm}
\noindent
{\sc Princeton University, 
Department of Mathematics,
Princeton, NJ 08540.} \hfill\break
{\tt smorel@math.princeton.edu}.

\vspace{2mm}
\noindent
{\sc University of Kentucky,
Department of Mathematics,
Lexington, KY 40506.}  \hfill\break
{\tt margaret.readdy@uky.edu}.

\end{document}